\documentclass[12pt,a4paper]{article}
\usepackage{amsfonts,amssymb}
\usepackage{graphicx, epsfig}
\newcommand{\lam}{\lambda}

\newtheorem{theorem}{Theorem}
\newtheorem{lemma}[theorem]{Lemma}

\newtheorem{prop}[theorem]{Proposition}

\newtheorem{rem}[theorem]{Remark}

\newenvironment{proof}[1][Proof]{\textbf{#1.} }
{\hfill\rule{0.5em}{0.5em}\medskip}
\newenvironment{proof*}[1][Proof]{\textbf{#1.} }{}

\def\optionalreading{$\bigstar$}

\begin{document}

\title{Jordan and Schoenflies in
non-metrical analysis situs \\}

\author{Alexandre Gabard and David Gauld\footnote{Supported by the Marsden Fund Council from Government funding, administered by the Royal Society of New Zealand.}
}

\maketitle



\newbox\quotation
\setbox\quotation\vtop{\hsize 10cm \noindent

\footnotesize {\it The Jordan curve theorem is the
mathematical formulation of a fact that shepherds have relied
on since time immemorial!}

\noindent Laurent Siebenmann, 2005 (in a letter to A.
Ranicki).

\smallskip
{\it Manifolds that are not paracompact are amusing, but they
never occur naturally. What is perhaps worse, it is difficult
to prove anything about them.}

\noindent Morris W. Hirsch \cite[p.\,32]{Hirsch_76}}

\hfill{\hbox{\copy\quotation}}


\bigskip
\bigskip
\newbox\abstract
\setbox\abstract\vtop{\hsize 12cm \noindent

\footnotesize \noindent\textsc{Abstract.} We show that both,
the Jordan curve theorem and the Schoenflies theorem extend to
non-metric manifolds (at least in the two-dimensional
context), and conclude by some dynamical applications \`a la
Poincar\'e-Bendixson.}

\centerline{\hbox{\copy\abstract}}

\bigskip
\newbox\MSC
\setbox\MSC\vtop{\hsize 12cm \noindent

}

\centerline{\hbox{\copy\MSC}}


\section{Introduction}

The intrinsic importance of {\it manifold theory},
regarded as the subdiscipline of topology climaxing the
venerable Euclidean geometry, can hardly be overemphasised in
view of
the
seminal achievements obtained over the
past two centuries (Gauss, Lobatschevskii, Riemann,
Poincar\'e, \dots, Perelman,
 \dots \ just to name a few). Especially elaborated is the metric theory where
it is postulated that a
distance function  generates the manifold topology. In
contradistinction the non-metric case (studied
by Cantor,
Hausdorff, Vietoris, Alexandroff, Pr\"ufer/Rad\'o, R.\,L.
Moore, Calabi--Rosenlicht, the Kneser family (Hellmuth and
Martin),
 M.\,E. Rudin, Zenor, Nyikos,
\dots)\footnote{Explicit references are
given in
\cite{BGG}, or
also in
\cite[especially those cited in \S 9]{Ullrich_00}.} remains,
comparatively, a
somewhat marginal branch,
whose
(in)significance is the object of recurrent cultural
controversies\footnote{Compare
Massey~\cite[p.\,47]{Massey_67}: ``{\it Such surfaces are
usually regarded as pathological, and ignored; \dots}'' Hirsch
\cite[p.\,32]{Hirsch_76}, as quoted above.
A less severe judgement is
Milnor~\cite[p.\,7]{Milnor_70}: ``{\it The main object of this
exercise is to imbue the reader with suitable respect for
non-paracompact manifolds.}'',
or Carath\'eodory \cite[p.\,707]{Caratheodory_1950}: ``{\it
[$\dots$] so da{\ss} unter Umst\"anden eine Triangulation
nicht existiert. Da{\ss} letzteres m\"oglich ist, zeigt er an
einem besonders lehrreichen Beispiel von Heinz Pr\"ufer.}''}.
Whatever the ultimate verdict should be,
it is fair to observe that many classical
notions like {\it dynamical systems} or {\it foliations} are
perfectly natural---at least well-defined---fields of
investigations even in the non-metric realm.
Some prolegomena towards a middlebrow foliation theory
developed over non-metric manifolds are to be found
in~\cite{BGG}. Concerning the allied
theory of
dynamical systems (from the viewpoint of {\it flows}, i.e.
continuous actions of the
real line~${\Bbb R}$), the
authors are
preparing
a modest paper~\cite{BGG3}
analysing which
among the basic principles of dynamics
permit an extension to non-metric mani\-folds.
(Topics include the Poincar\'e-Bendixson theory, G.\,D.
Birkhoff's minimal systems, the Whitney-Bebutov theory of
cross-sections and flow-boxes, Whitney's flows\footnote{i.e.,
a flow
parameterising the leaves of a given orientable
one-dimensional foliation.}, the construction of
transitive\footnote{Following G.\,D. Birkhoff, a flow is
{\it transitive} if it has at least one dense orbit, and
{\it minimal} if all orbits are dense.} flows \`a la
Sidorov/Anosov-Katok, Anatole Beck's technique for slowing
down flow lines.)
Such dynamical motivations
led us to
inquire about the
availability of the Jordan and Schoenflies theorems without
any metrical proviso, which is the
chief concern of the present note. (This hopefully  justifies
our somewhat old fashioned title,
winking at Veblen's 1905 paper \cite{Veblen_1905},
often
regarded as
the first rigorous
proof of the {\it Jordan curve theorem}---abbreviated as (JCT)
in the sequel.)

If we are permitted to give a
slight
refinement of Hirsch's
phraseology above, we
believe that non-metric manifold theory
takes in reality
a two-fold incarnation: there is a ``soft-side'', usually
concerning compact, or even Lindel\"of subobjects and a
``hard-side''
involving the whole manifold itself, and which typically is
{\it not safe from the invasion of set-theoretic independence
results}\footnote{Formulation borrowed from
Nyikos~\cite[p.\,513]{Nyikos_82}. The chief issue is that the
answer to an old question of Alexandroff--Wilder relative to
the existence of a non-metric {\it perfectly normal} manifold
(i.e. each closed subset is the zero-set of a continuous
real-valued function) turned out to be independent of the
usual axiomatic ZFC (Zermelo--Fraenkel--Choice).}.
%
The
Jordan and Schoenflies problems of this note
belong
to the
former class of easy problems, reducible to metrical knowledge
as we shall see.
A similar metrical
reduction occurs with flows,
since letting flow any chart\footnote{By a {\it chart} we
shall mean an open subset homeomorphic to ${\Bbb R}^n$.} $V$
one generates a Lindel\"of submanifold
$f({\Bbb R} \times V)$,
to which one may apply the
(metric) Whitney-Bebutov theory of cross-sections and
flow-boxes. This
extends the availability of the Poincar\'e-Bendixson theory,
as well as
the fact that non-singular flows
induce foliations.
The
Lindel\"ofness
of $f({\Bbb R} \times V)$ also shows that {\it non-metric
manifolds
never support minimal flows}.
Accordingly, it is sometimes
much easier to prove things about non-metric than metric
manifolds, as corroborated
by the
elusive {\it Gottschalk conjecture}
on the (in)existence of a minimal flow on a ``baby'' manifold
like ${\Bbb S}^3$. In contrast Hirsch's statement remains
perfectly
vivid when it comes to
the existence of smooth structures on 2-manifolds, where it is
still much
undecided in which category ``soft vs. hard''  this
problem will ultimately fall.
Recall that the similar
question for PL
structures was recently solved by Siebenmann~\cite[{\sc
Surface Triangulation Theorem} (STT), p.\,18--19]{Sie}.

Perhaps  another motivation
for a
non-metric version of Schoenflies, arises  in the context of
the {\it Bagpipe Theorem} of Nyikos~\cite[Theorem 5.14,
p.\,666]{Nyikos_84},
a far reaching ``generalisation'' of the classification
of compact surfaces, extended to
\hbox{$\omega$-bounded}\footnote{A topological space is said
to be {\it $\omega$-bounded} if the closure of any countable
subset is compact.} surfaces. To be
honest the latter is rather a ``structure theorem''
as the
tentacular {\it long pipes} emanating from the compact {\it
bag} (a compact bordered surface) may exhibit a bewildering
variety of
topological types.
Understanding long pipes is tantamount to describing
simply-connected $\omega$-bounded surfaces, via the canonical
bijection given by  ``filling the pipe with a disc'', whose
inverse operation is ``disc excision''.
In this context, it may be observed that Nyikos (cf.
\cite[p.\,668, \S 6]{Nyikos_84}) relies on an
ad hoc definition of
simple-connectivity  which is
a consequence of the non-metric Schoenflies theorem
(Propositions~\ref{generalised_schoenflies} and \ref{converse}
below). Hence
our results
just bridge a little gap between the conventional definition
of
simple-connectivity (in terms of the vanishing of the
fundamental group $\pi_1$) and the one
adopted by Nyikos (separation by each embedded circle, with at
least one residual component having compact closure).
Section~6 of Nyikos~\cite{Nyikos_84}
shows that even
under the stringent
assumptions $\pi_1=0$ jointly with $\omega$-boundedness (which
should be regarded as a non-metric pendant of compactness)
two-dimensional topology
permits a
menagerie of
specimens.

The most naive approach, say to the Schoenflies problem, could
be the following: given a {\it Jordan curve} $J$ (i.e. an
embedded circle)  in a non-metric simply-connected surface,
try to engulf $J$ in a chart to conclude via the classical
{\it Schoenflies theorem}---henceforth abbreviated as (ST)---
that $J$ bounds a disc. This is
somewhat hazardous
because all the given data (as well  the ambient surface as
the embedded circle) are a priori extremely large (about the
size of
an expanding universe). However a
refinement of this idea is successful:
cover the
range of a null-homotopy by a finite number of charts, the
union of which provides a metric subsurface into which $J$ is
contractible to point, hence bounds a disc (by a homotopical
version of (ST), cf. Section~\ref{Baer-Epstein} for the
details).

Beside this ``geometric approach'' there is a more
``algebraic'' one
relying on {\it singular homology}\footnote{Initiated by
Lefschetz, 1933 and taking its definitive form
with Eilenberg, 1944.}, whose
intervention is
prompted by the fact that
non-metric manifolds
are
inherently intriangulable\footnote{The
easy argument (going back at least to
Weyl~\cite[p.\,24]{Weyl_13})
is that when enumerating simplices by adjacency, one sees that
a triangulated (connected) manifold consists of at most
countably many simplices, so is $\sigma$-compact, hence metrisable.}.
We are still hesitant about deciding which of the two
approaches provides more insights, so we decided to include
both.
A
useful reference for the singular homology of manifolds is the
paper by Samelson~\cite{Sa}, of which we shall need
 the basic {\it vanishing result for the top-dimensional
homology of an open
(connected,
Hausdorff) manifold.}
The surprising
issue
is that no ``extra-terrestrial'' non-metric
``geometric topology'' is required, just easy algebra and
finistic topology (classification of compact surfaces) do the
job. [This is a
fair judgement, modulo the
fact
that already for the non-metric Jordan theorem, our proof
has a
reliance on
(ST), so
fails to be
``pure homology''.] At the end of the note we
present a converse to the non-metric
(ST) (Proposition~\ref{converse}).

Combining both
results~(Propositions~\ref{generalised_schoenflies} and
\ref{converse}) we can state our main result as:

\begin{theorem} A (Hausdorff) surface $M$ is simply-connected
if and only if
each Jordan curve in $M$ bounds a $2$-disc in $M$.
\end{theorem}

\smallskip
{\small {\sc
Typo- and bibliographical
conventions.} We shall
put in small fonts
certain digressions not directly
relevant to our main purpose (those optional readings, marked
by the symbol \optionalreading, can be
omitted without loss of continuity). Many classical references
related to Jordan and Schoenflies are listed in
Siebenmann~\cite{Sie}; so any
lazy
referencing, by us, of the form [Jordan, 1887] means that the
item can
be located  in Siebenmann's bibliography.
On the
other hand we try to reserve
(hopefully not too caricatural) historical comments to
footnotes in order
to keep clean the logical structure of the argument. Those
historical details are
provided as
distractions,
which
in the best cases represent only a first-order approximation
toward a sharpened picture provided by
the first-hand sources. Perhaps the
diagram below provides a snapshot view of some of the
historical background relevant to our purpose.

}

\begin{figure}[h] \centering
    \epsfig{figure=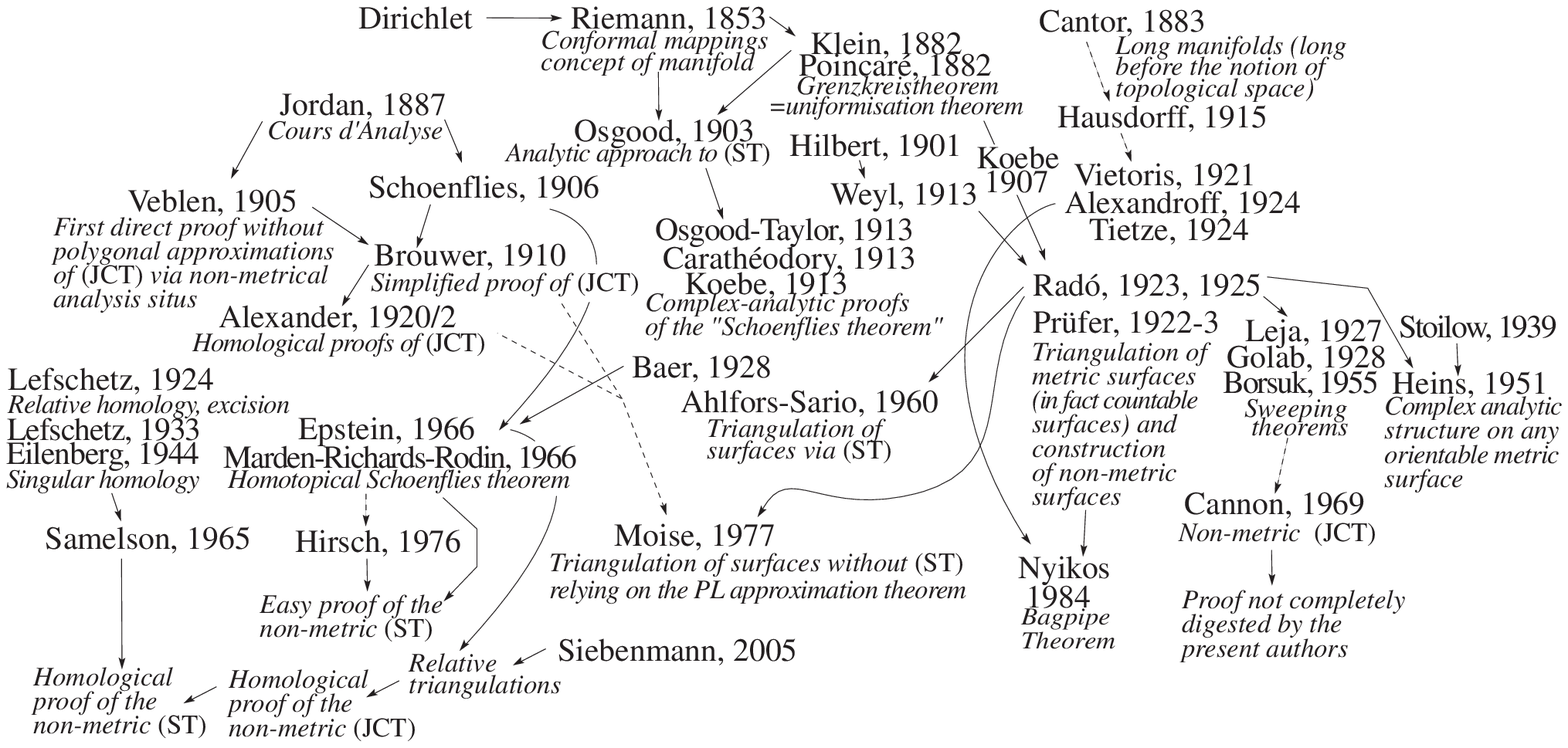,width=162mm}
\vskip-35pt\penalty0
\end{figure}




\section{The geometric
approach}\label{Baer-Epstein}

Before presenting
the homological argument, we discuss two alternative
approaches that were suggested to us by a closer look to the
existing literature.

\smallskip

(1)
We realised that the non-metric Jordan theorem is also stated
by R.\,J. Cannon~\cite[Remark on p.\,97]{Cannon_69}: ``{\it
The Jordan curve theorem, which is well known in the case of
the plane, is true for noncompact simply connected
$2$-manifolds in general.}''
Cannon's proof is rather
succinct
in its reliance on
a   ``sweeping'' theorem of Borsuk,
to which alas
no explicit reference is provided. The most appropriate
reference we were able to locate is Borsuk~\cite{Borsuk},
where
previous works of [Leja, 1927]\footnote{Leja's original
argument depends
on a parametric version of the Riemann mapping theorem due to
[Rad\'o, 1923]=\cite{Rado_23_Szeged}, and so will
doubtfully satisfy the ``topologically inclined'' reader.}
and [Go{\l}ab, 1928] are revisited.
(Go{\l}ab is apparently responsible for
the terminology {\it balayage}=sweeping.) Unfortunately we
failed to understand the details of Cannon's proof,
as in all sweeping theorems we are aware of (\cite{Borsuk} and
the two subreferences cited above), the ambient manifold
is ${\Bbb R}^n$, a
condition which seems hard to fulfil when covering by charts
the range of a null-homotopy.


\smallskip

(2) Now we come to the geometric approach to the (non-metric)
Schoenflies problem, which strangely enough permutes the
``logical r\^oles'' of Jordan and Schoenflies. It is based on the following exercise in
Hirsch's
book~\cite[p.\,207, Ex.\,2]{Hirsch_76},
stated as: ``{\it Let $M$ be a surface\footnote{Presumably
assumed paracompact, cf.  the Convention formulated in
\cite[\S 5, p.\,32--33]{Hirsch_76}.} and $C\subset M$ a
circle. If $C$ is contractible to a point in $M$ then $C$
bounds a disk in $M$}''.
Call this statement (HST) for {\it homotopical Schoenflies
theorem}. The pleasant issue is that this statement
immediately transcends itself beyond the metric realm:

\begin{prop}\label{NMHST} Let $M$ be a
(Hausdorff not necessarily metric) surface and $C$ be a
null-homotopic Jordan curve on $M$. Then $C$ bounds a $2$-disc
in $M$. (In particular if $M$ is simply-connected, each Jordan
curve bounds a disc.)
\end{prop}

\begin{proof} By compactness we may cover the image of a
contracting homotopy shrinking $C$ to a point by a finite
number of charts (open sets homeomorphic to the plane ${\Bbb
R}^2$). So $C$ is contained in a certain Lindel\"of (hence
metric) subsurface $M_{\ast}$ into which $C$ is
null-homotopic\footnote{This easy argument is {\it not} new
and appears in Cannon ({\it loc. cit.}
\cite[p.\,97]{Cannon_69}).}.
By (HST) it follows that $C$ bounds a disc in $M_{\ast}$,
which can of course be regarded as embedded in $M$.
%
%
\end{proof}

Covering space theoretic
proofs of (HST) are proposed in Epstein~\cite[Theorem 1.7,
p.\,85]{Epstein_66} and in
Marden--Richards--Rodin~\cite{Marden_66}\footnote{These
authors assume orientability, which
is not required (at least for the first part of their
statement).}, following a method that goes back at least to
Baer~\cite[\S 2, (b), p.\,106--107]{Baer_1928}\footnote{Baer
assumes his surface $\frak F$ closed of genus $g>2$, but his
argument
adapts easily to the general case.}. The idea is
simply to lift the problem to the universal covering. So one
needs first knowledge\footnote{For Poincar\'e
\cite[p.\,114]{Poincare_1883}, this seems to be obvious a
priori (i.e., prior to uniformisation), as he writes: ``{\it
La surface de Riemann est [\dots] simplement connexe et ne
diff\`ere pas, au point de vue de la G\'eom\'etrie de
situation, de la surface d'un cercle, d'une calotte
sph\'erique ou d'une nappe d'un hyperbolo\"{\i}de \`a deux
nappes.}''} of simply-connected (metric) surfaces:

\begin{lemma}\label{simply-connected} A simply-connected
metric surface is homeomorphic either to the plane ${\Bbb
R}^2$ or to the sphere ${\Bbb S}^2$.
\end{lemma}

\begin{proof} Of course the statement is a direct consequence
of the uniformisation theorem of Riemann surfaces
(Klein-Poincar\'e-Koebe 1882--1907); after using Rad\'o
\cite{R} to triangulate and Heins \cite{Heins_51} to introduce
a ${\Bbb C}$-analytic structure. As an alternative elementary
approach one can
appeal to a certain enumeration scheme for the triangles of an
open triangulated simply-connected 2-manifold proposed by van
der Waerden \cite{van_der_Waerden_1939} and
Reichardt~\cite{Reichardt_1941}: one can successively
aggregate triangles while never introducing a triangle
$\Delta_k$ touching the earlier aggregate $\Delta_1 \cup \dots
\cup \Delta_{k-1}$ along one edge plus its opposite vertex.
Such an enumeration
allows one to
construct a homeomorphism with the plane ${\Bbb R}^2$
by considering a suitable subsequence forming an
ascending chain of closed discs. One can also conclude with
the monotone union theorem of Morton Brown~\cite{Brown_1961}.
[In fact van der Waerden's motivation (cf. also the
enthusiastic paper by Carath\'eodory \cite{Caratheodory_1950})
was to provide the ``simplest possible''
foundations
to the uniformisation theorem within the frame of pure complex
function theory, while avoiding any potential-theoretic
intrusion\footnote{To appreciate fully this fact, we may refer
to Gray's historical survey~\cite[\S 6, p.\,78]{Gray_1994}}.]
Other polyhedral proofs
are to be found in Ahlfors-Sario~\cite[\S 44D, p.\,104]{AS} or
in Massey~\cite[p.\,200, Ex.\,5.7]{Massey_67}.
(We are not aware of an
argument
bypassing Rad\'o's triangulation theorem.)
\end{proof}

\begin{lemma} {\rm (HST)} \label{Baer_thm} (Baer, 1928, H.\,I. Levine, 1963,
Epstein, 1966, Marden et al. 1966). Let $M$ be a metric
surface and $C$ be a null-homotopic Jordan curve on $M$. Then
$C$ bounds a $2$-disc in $M$.
\end{lemma}

\begin{proof}
We
follow Baer
\cite[\S 2, (b), p.\,106--107]{Baer_1928}.
Lift the Jordan curve $C$ to the
universal covering $\pi\colon\widetilde{M}\to M$ to obtain
another Jordan curve
$\Gamma$
(in view of the
null-homotopic assumption). By the topological avatar of the
Poincar\'e-Volterra theorem (see Bourbaki \cite[Chap.\,I, \S
11.7, Corollary~2, p.\,116]{Bourbaki} or
\cite[p.\,197]{Guenot-Narasimhan}), the
cover $\widetilde M$ is second countable, hence
metric. Therefore $\widetilde{M}$ is either ${\Bbb R}^2$ or
${\Bbb S}^2$ by
Lemma~\ref{simply-connected}\footnote{In Epstein's proof this
classification is not taken for granted, but rather deduced
from (HST). However this is done at the cost of an
Appendix
on PL technology.
}, and it follows that
$\Gamma$ bounds a disc $\Delta$ in $\widetilde{M}$ by the
classical
(ST).  It is enough to show that the disc $\Delta$
maps
bijectively to $M$ via the covering projection $\pi$, yielding
the desired disc $D=\pi(\Delta)$ bounding $C$. A priori, three
collapsing possibilities could occur:

(1) Two points of the boundary of $\Delta$ are {\it
equivalent} (under a deck-transformation); [this case is
obviously impossible]

(2) A point interior to $\Delta$ and one lying on its boundary
$\Gamma$ are equivalent;

(3) Two interior points $p_1,p_2$ of $\Delta$ are equivalent.

Case (2) implies that in the interior of $\Gamma$ there would
be  a nested sequence of curves lying over $C$, hence an
accumulation point (compactness of $\Delta$); violating the
discreteness of the fibers of the (universal) covering map.
Baer
argues that distinct lifts of $C$ are disjoint,
as otherwise
projecting down to $M$ one would get ``multiple points'' on
$C$ (contradicting $C$ being a simple curve). [This
disjunction property can also be seen by considering the
restricted covering $\pi\colon \pi^{-1}(C) \to C$, whose total
space is a disjoint union of circles, each simply covering the
base circle.]
By (JCT), in the plane a Jordan curve possesses a unique
(compact) {\it inside}, by which we
mean the interior of the Jordan curve plus the Jordan curve
itself.
By (ST) one has an alternative argument to Baer's,
as the deck-transformation $\gamma$ taking $p_1\in \Gamma$ to
$p_2\in {\rm int} (\Delta)$ would map the disc $\Delta$ into
itself,
violating the Brouwer fixed-point theorem.
In both arguments,
it must be remarked that the deck-transformation $\gamma$
carries $\Gamma$ to the curve $\gamma(\Gamma)$ which is
interior to $\Gamma$ (by disjunction plus (JCT)), hence since
$\gamma$ is a global homeomorphism of the plane it must take
the inside of $\Gamma$ to the inside of $\gamma(\Gamma)$.
As a Jordan curve $J_{0}$ contained in the inside of a Jordan
curve $J$ has an inside contained in the inside of $J$ (cf.
Siebenmann~\cite[{\sc Jordan Subdomain Lemma}, p.\,4]{Sie}),
it follows that $\gamma(\Delta) \subset \Delta$,
contradicting either Brouwer, or yielding the infinite
sequence $p_1,\gamma(p_1), \gamma^2(p_1),\dots$ in the
compactum
$\Delta\supset
\gamma(\Delta)\supset \gamma^2(\Delta)\supset \dots$,
corrupting the discreteness of the fibres, as argued by Baer).

Case (3) reduces to Case (2). Indeed choose an
arc\footnote{A {\it Jordanbogen} in
Baer~\cite[p.\,107]{Baer_1928}:
this follows either from (ST) or more elementarily by
a clopen argument that shows that in any connected manifold
(whether metric or not), one can join any two given points by
an embedded arc, cf.
\cite[Prop.\,1]{Gauld_2009}.} $A$ inside ${\rm int} (\Delta)$
joining $p_1, p_2\in {\rm int} (\Delta)$, then its projection
is a closed curve\footnote{{\it \"uber einer geschlossenen
Kurve $\frak K$}; it is
not perfectly clear if Baer assumes his curve to be simple?}
(=loop) $K$ on $M$, which is not null-homotopic (as $p_1\neq
p_2$). Considering successive lifts $A=A_1, A_2, \dots$ of
$K$, where $A_i$ starts from the end-point of its predecessor
$A_{i-1}$, shows that we will eventually leave
$\Delta$ (else
there would be
again a corruption of the discreteness of the fibre). Of
course this holds
in the 
absence of a periodic motion, which might be inferred from
the torsion-free property of the $\pi_1$ of aspherical
manifolds. 
Alternatively, one may argue that a cyclic pattern leads to a
string $J:=A_1\cup \dots \cup A_k$ of $A_i$'s which is a
closed curve. Can we arrange it to be Jordan, i.e. simple? If
not,
there would be double points when projecting down.  As yet we
have not ensured that $K$ is a simple closed curve, however a
simple trick is to travel along the arc $A$ and as soon as its
projection down to $M$ exhibits a self-intersection, we may
cut out a subarc $A_0\subset A$ projecting
to an
embedded circle; and redefine $A$ as $A_0$.  Then
a fixed point for $\gamma$, the deck-translation induced by
the loop $K$, is created
in the inside
of the
Jordan curve $J$  by the Brouwer fixed-point theorem.
This justifies the absence of ``periodicity'', so that the
end-point of $A_n$ is not in $\Delta$ (for some sufficiently
large integer $n$). Then
by
(JCT), the path $A_1\cup A_2\cup \dots \cup A_n$ will meet
$\Gamma$. Since $A$
does not meet $\Gamma$, the
intersection point $p\in A_i \cap \Gamma$ occurs on a
later arc $A_i$ ($i\ge 2$), and therefore the pair consisting
of $p$ and its deck-translation
back to $A$ satisfies the requirement of Case (2).
%
%
\end{proof}

We may
now move
towards Jordan,
perhaps first recalling
the following
fine words of Felix Klein \cite[p.\,531]{Klein_1882} in 1882,
five years before\footnote{Of course this does not discredit
Jordan's priority as Klein is always perfectly clear (if not
vindicating) the merely heuristic value of his exposition,
primarily intended to be a diffusion of Riemann's ideas.
} the official ``proof'' of [Jordan, 1887]:
``{\it [\dots] da{\ss} die jetzt betrachtete Kurve, gleich
einer solchen,
die sich in einen Punkt zusammenziehen l\"a{\ss}t,
die gegebene Fl\"ache in getrennte Gebiete zerlegt}''. This
contains
(more-or-less) the following statement:

\begin{prop} (Klein, 1882)
Let $C$ be a null-homotopic Jordan curve on a Hausdorff
surface $M$ (metric or not).
Then $C$ divides $M$ (i.e. $M-C$ is
disconnected).
\end{prop}

\begin{proof} By Proposition~\ref{NMHST}, $C$
bounds a disc $D$ in $M$.
Its interior $U={\rm int} (D)$ is contained in $M-C$, and
cannot be enlarged without meeting $C$ (by
the Frontier Crossing Lemma as
stated in Siebenmann~\cite[p.\,2]{Sie}). It follows that $U$
is a connected component of $M-C$,
but not the unique one as otherwise $U=M-C$, so that when
taking closures $D=M$,
violating the
assumption that $M$ is a genuine surface (without boundary).
\end{proof}

We may observe that in this geometric approach, the historical
as well as the logical order of Jordan and Schoenflies gets
reversed.
The sequel of the
paper presents an alternative ``algebraic'', indeed
homological approach, where
the natural order is restored.

\section{Generalised (non-metric) Jordan
theorem}\label{JCT}

We now start with a general
(i.e. not metrically confined) formulation of the Jordan
curve/separation theorem\footnote{See for the classical case
[Jordan, 1887], [Veblen, 1905], etc.,
and for a panoramic view
\cite{Sie}.}:

\begin{prop}\label{generalised_Jordan} (Generalised
Jordan curve theorem) Let $M$ be a (connected) Hausdorff
simply-connected surface. Then $M$ is
\emph{dichotomic}\footnote{The
term {\it schlicht(artig)} is also employed, but in the
non-metric context it would be misleading.
}, i.e. each
embedded circle $J$ in $M$ divides the surface into exactly
two
components. Moreover the (topological) frontier
of each
component of $M-J$ is $J$.
\end{prop}

\begin{proof} {\small Note first that the Hausdorff axiom is
essential: without it take $M=B\times {\Bbb R}$ the product of
the {\it branching line} $B$ (as defined e.g. in
\cite[Figure~1]{Baillif_Gabard_08}) with the usual line, on
which it is easy to draw a non-dividing circle.

}

Our proof
relies
on the following geometric lemma\footnote{In contrast to the
homological proofs of the Jordan curve
theorem (say by Brouwer~\cite{Brouwer_10}, or more fairly
[Alexander, 1920], [Alexander, 1922]=\cite{Alexander_1922})
 where the ambient manifold is
known, either a Euclidean ${\Bbb R}^n$ or a sphere ${\Bbb
S}^n$, we need here more geometric control, furnished by the
classical
(ST), in order to construct a tube around
the Jordan curve. Hence (at least in our presentation) the
non-metric Jordan theorem does {\it not} boil down to pure
homology theory.}:

\begin{lemma} \label{tube} (Tubular neighbourhoods of circles) Let $J$ be a
Jordan curve (=a homeomorph of the circle) in a Hausdorff
orientable surface $M$. Then there is an open set $T$ in $M$
containing $J$ with a homeomorphism of pairs $(T, J)\approx
({\Bbb S}^1\times{\Bbb R},{\Bbb S}^1\times \{ 0 \})$.
\end{lemma}

\begin{proof} Since $J$ is compact, it can be covered by
finitely many charts of $M$, so that we can reduce to the case
where the ambient manifold $M$ is Lindel\"of (hence metric).
Then classical results do the work: indeed by Rad\'o~\cite{R}
metric surfaces can be triangulated, and then the required
tubular neighbourhood might be constructed via combinatorial
methods. Finally the trivial product structure of $T$ comes
from orientability (as opposed to a twisted ${\Bbb R}$-bundle
over ${\Bbb S}^1$, i.e. a M\"obius band, which would violate
it).

In fact the metric subsurface $M_{\ast}$ engulfing the Jordan
curve $J$
(of the previous paragraph) is indeed triangulable but one
must
ensure that {\it a triangulation of
$M_{\ast}$ can be
arranged in such a way that $J$ is a subcomplex.} (In this
situation the required tube $T$
may be constructed via
{\it regular neighbourhood theory}\footnote{Initiated by
[J.\,H.\,C. Whitehead, 1939], developed by [Zeeman, 1962],
etc. cf. also [Rourke-Sanderson, 1972].}.)
A priori, the existence of such a triangulation looks
fragile, especially in view of fractal curves or the
construction by [Osgood, 1903] of a Jordan curve in the plane
${\Bbb R}^2$
of positive Lebesgue measure. However it is precisely the
content of
(ST), to
ensure
that whatever the complexity of a Jordan curve $J$ in ${\Bbb
R}^2$ might be, there is
still a global homeomorphism of the plane taking $J$ to the
unit circle ${\Bbb S}^1$ (or to a
triangle). In
particular, {\it there exists a triangulation of the plane
${\Bbb R}^2$ such that
any given Jordan curve $J$ occurs as a subcomplex.}
Alternatively, {\it one finds a tube around any $J$}, just by
pulling back a neat annulus around ${\Bbb S}^1$.
In our situation nothing ensures that $M_\ast$ is a homeomorph
of ${\Bbb R}^2$.
However
the classical
(ST) allows one to solve
the corresponding global problems (i.e. ${\Bbb R}^2$ replaced
by an arbitrary metric surface $M$):

\begin{lemma} \label{Siebenmann} {\rm (i)} {\sc Relative surface triangulation theorem.}
{\rm (Compare Siebenmann \cite[Remark (b), p.\,19]{Sie})}
Given a pair $(M,\Gamma)$ consisting of a graph $\Gamma$
(locally finite simplicial complex of dimension $1$) embedded
as a closed subset of a metric surface $M$, one can construct
a triangulation of $M$ so that $\Gamma$ occurs as a
subcomplex. {\rm [In our setting we just need the case where
$\Gamma$ is a circle, which is treated in
Epstein~\cite[Appendix]{Epstein_66}]}

{\rm (ii)} {\sc Tubular neighbourhoods of graphs.} {\rm (Cf.
again \cite[Remark (b), p.\,19]{Sie})} Same data as in {\rm
(i)}, one can (directly) construct a tubular neighbourhood of
the graph $\Gamma$. {\rm The proof uses the {\it local graph
taming theorem} (cf. Siebenmann~\cite[\S 9, p.\,16]{Sie})
conjointly with the collaring
techniques of Morton Brown.
}
\end{lemma}
%
%
%
This lemma completes the proof of Lemma~\ref{tube}.

\smallskip
{\sc  \optionalreading Variant ({\rm \`a la Edwin Evariste}
Moise):} Alternatively it is certainly possible to establish
both these results without relying on
(ST), following the techniques employed in Moise~\cite{M}, who
is able to prove the (absolute) surface triangulation theorem,
without
reference to (ST),
by founding everything on the PL approximation theorem.
\end{proof}

 {\small {\sc  \optionalreading Historical digression on the
triangulation of surfaces (Rad\'o, Pr\"ufer, 1922--1925).}
(For much sharper reports compare \cite[especially \S 6, \S
9]{RemSch_97} and \cite[End of \S 8 and \S 9]{Ullrich_00}.) On
pages 110-111
of his paper, Rad\'o~\cite{R} 
recalls that the triangulation theorem
was in
special cases treated by [Weyl, 1913]=\cite[p.\,21,
p.\,32]{Weyl_13} (the case of {\it analytische
Gebilde}=concrete Riemann surfaces
arising
via analytic continuation of a holomorphic function-germ), and
respectively [H.~Kneser, 1924]
(triangulability in
presence of a {\it Kurvenschar} (=foliation) on a compact
surface). It is interesting to
observe that Rad\'o's 1925
proof (cf. \cite[Hilfssatz 2, p.\,111--114]{R}) does not seem
to use
(ST)\footnote{Of course this is
not so much of a surprise if one recalls from Siebenmann
\cite[\S 4, {\sc Historical notes}]{Sie} that the
``Schoenflies theorem''
appellation seems to have been coined only in [Wilder,
1949].}. In Remmert-Schneider \cite[p.\,188, \S 9]{RemSch_97}
(where the early interactions between
Pr\"ufer and
Rad\'o are beautifully commented on), it is
asserted that Rad\'o's 1925
proof ``{\it benutzt den Riemannschen Abbildungssatz}'', i.e.
relies
on the Riemann mapping theorem. [Here we are not
 sure to
agree completely with Remmert--Schneider's
assertion, and we believe instead that Rad\'o gives himself
much pain to work at a purely topological level.]
As early as 1923, Rad\'o \cite[p.\,35--36]{Rado_23} presents a
proof of the triangulation theorem (at least for Riemann
surfaces
based on the {\it Grenzkreistheorem} of Klein-Poincar\'e). For
the general case he
provides only a sketch \cite[p.\,37]{Rado_23}.
Unfortunately, it seems that Rad\'o
condensed his 1923 exposition
influenced by
the erroneous suggestion
of
Pr\"ufer
that triangulations
could
exist without
any countability proviso. Retrospectively this early mistake
of Pr\"ufer
looks
astonishing in view of the long (line) manifolds of [Cantor,
1883, Hausdorff, 1915, Vietoris, 1921, Tietze\footnote{To whom
Hausdorff communicated his 1915 construction.}, 1924 and
Alexandroff, 1924]\footnote{Accurate references located in
\cite{RemSch_97}, \cite[\S 9]{Ullrich_00} where this ``long''
string
of (re)discoverers of the long ray is
carefully
documented (including the ``recent'' discovery by E. Brieskorn
and W. Purkert in the Univ. Bibl. at Bonn of an unpublished
Nachla{\ss} of F. Hausdorff, dated in 1915.) Meanwhile this
Nachla{\ss} has been published in~\cite{Hausdorff_1915}.
}, but
turned out to be
extremely
fruitful,
by leading
to
a new generation of non-metric manifolds, the so-called {\it
Pr\"ufer manifold(s)}.
(The latter was first described in print by Rad\'o~\cite{R},
but already mentioned in \cite[p.\,35, footnote~9]{Rado_23}.)
From a
strict logical viewpoint, it looks
intriguing to
question the rigour
of Rad\'o's 1925 proof
(a naive minded objection being that while scanning through
Rad\'o's
argument
one does not encounter any citation to Schoenflies nor to
Osgood, but perhaps such a use is implicit somewhere in
Rad\'o's proof).
Such a moderate criticism of Rad\'o
seems  also implicit in Remmert--Schneider's formulation
\cite[p.\,187]{RemSch_97}: ``{\it Eine heutigen Ma{\ss}t\"aben
gerecht werdende Behandlung des Triangulierungsproblems
\dots}'', where of course they refer to the proof
presented by Ahlfors-Sario \cite[Chap.\,1,
\S 8, p.\,105--110]{AS} (this is perhaps the first
place where
a reliance
on
(ST) for triangulability is made explicit)\footnote{Moreover
it seems that the Ahlfors-Sario proof benefited from some
corrections pointed out by G.~Thomas (compare page $vi$ of the
preface of the 1965 Second Printing of \cite{AS}).}. It
should, however, be emphasised that the Ahlfors-Sario proof
stays very close to the 1925 proof of Rad\'o. Other proofs
(usually restricted to the compact case) are given in
[Doyle-Moran, 1968]=\cite{Doyle-Moran} and [Thomassen, 1992].
In the latter reference there is (on page 116) a (too?)
severe criticism that the previous proofs (of triangulability)
relied on geometric intuition. In sum,
available proofs of the triangulability of metric  surfaces
(non-compact case included)
include the following list (in chronological order): [Rad\'o,
1925]=\cite{R} (with a sketchy precursor in \cite{Rado_23}),
[Ahlfors-Sario, 1960]=\cite{AS}, [Moise,
1977]=\cite[p.\,60]{M} and
[Siebenmann, 2005]=\cite{Sie}. It is to be noted that the
proof in [Moise, 1977] does {\it not}
depend on
(ST); and in \cite[p.\,62]{M} one even finds a serious ``3D''
justification: ``{\it Ordinarily, the triangulation theorem
for $2$-manifolds is deduced from the Sch\"onflies theorem.
This method may be simpler,
once the Sch\"onflies theorem is known, but it is in a way
misleading. In dimension $3$, the Sch\"onflies theorem fails,
but the triangulation theorem still holds. Thus we should
avoid creating the impression that the latter depends on the
former.}''
In the same vein, it can be observed that in our
non-metric two-dimensional
context
the ``reverse situation'' occurs:  the Schoenflies theorem
holds, but the triangulation theorem fails (dramatically).

}

\medskip

{\small \optionalreading {\sc A long standing
question of Spivak and Nyikos.} It is a
natural
problem to wonder if any surface admits a smooth
structure (cf. Spivak~\cite[page A-18]{Spivak}:
``{\it I do not know whether every $2$-manifold has a
$C^{\infty}$ structure.}'' and
Nyikos~\cite[p.\,108]{Nyikos_93}: ``{\it Are there
$2$-manifolds
and $3$-manifolds that do not admit smoothings?}''). In the
metric (two-dimensional) case the answer is
{\it positive} either by using methods of Riemann surface
theory\footnote{\label{Heins} Cf. M. Heins \cite{Heins_51},
who (improving
works of Stoilow) shows
the existence of a complex-analytic structure on any metric
orientable surface; hence via the two-fold orientation
covering trick, one gets a DIFF structure on any metric
surface (in reality one gets much more, namely a so-called
``Klein surface''
or ``dianalytic structure'', much studied by Alling-Greenleaf,
etc.).
%
%
} or by
``softer'' DIFF methods, albeit an explicit reference seems
difficult to locate (as deplored by
Remmert-Schneider~\cite[p.\,190]{RemSch_97}:
``{\it Erstaunlicherweise scheint hierf\"ur kein direkter
Beweis in der Literatur zu existieren.}''
Does Siebenmann's
existence of a PL structure
for non-metric surfaces (cf. \cite[
p.\,18--19, proof of (STT)]{Sie})
bring us closer to a positive answer to the Spivak-Nyikos
existence question? 

}

\medskip
{\bf Finishing the Proof of
Proposition~\ref{generalised_Jordan}.} Once a tube $T$ around
$J$ is available, the proof reduces to homological routines
(exact sequence of a pair plus excision). The sequence of the
pair $(M, M-J)$ reads (coefficients are taken in ${\Bbb Z}$
and subscripts are the ranks of the homology groups, whose
finiteness will be soon~evident):
$$
H_1(M)\to H_1(M,M-J)_s\to H_0(M-J)_r \to H_0(M)_1 \to
H_0(M,M-J)\,.
$$
Both groups at the extremities vanish (recall that the first
integral homology is the abelianisation of the fundamental group).
Therefore $r=s+1$ (additivity of the rank). We shall use
excision
to compute $s$ (cf. e.g.
\cite[Thm 2.11, p.\,47]{Vi}). By excising the complement of
the tube $T$ from the pair $(M, M-J)$, we
get an isomorphism $H_1(T,T-J)\approx H_1(M, M-J)$.
In turn we may interpret  $(T, T-J)$ as the result of excising
the two poles of a \hbox{2-sphere} $({\Bbb S}^2, J)$, where
$J$ is standardly embedded as the equator; yielding an
isomorphism $H_1(T,T-J)\approx H_1({\Bbb S}^2, {\Bbb S}^2-J)$.
Writing the sequence of the pair $({\Bbb S}^2, {\Bbb S}^2-J)$
as: $0=H_1({\Bbb S}^2)\to H_1({\Bbb S}^2, {\Bbb S}^2-J)_s\to
H_0( {\Bbb S}^2-J)_2\to H_0({\Bbb S}^2)_1\to H_0({\Bbb S}^2,
{\Bbb S}^2-J)=0$, we see that $s=1$. Hence $r=2$, completing
the proof that $M$ is dichotomic.

Alternatively using reduced homology, as $\tilde
H_0(M)=0=\tilde H_0(\Bbb S^2)$, we obtain isomorphisms
$$\tilde H_0(M-J)\approx H_1(M,M-J)\approx H_1(T,T-J)\approx H_1(\Bbb S^2,\Bbb S^2-J)\approx\tilde H_0(\Bbb S^2-J)\approx\Bbb Z.$$

The last clause follows easily from the existence of the tube
$T$.
\end{proof}

\section{Generalised (non-metric) Schoenflies
theorem}

\begin{prop}\label{generalised_schoenflies} (Non-metric
Schoenflies theorem) Let $M$ be a Hausdorff simply-connected
surface. Then $M$
is \emph{Schoenflies}\footnote{{\it Irreducible} would perhaps
be a more neutral terminology.}, i.e. each embedded circle $J$
in $M$ bounds a $2$-disc in~$M$.
\end{prop}

\begin{proof} No loss of generality results
in assuming $M$ to be connected. If $M$ is metric, then
Lemma~\ref{simply-connected} implies that $M$ is either ${\Bbb
R}^2$ or ${\Bbb S}^2$, and the conclusion is given by the
classical
(ST)\footnote{Compare [Schoenflies, 1906] ``versus'' [Osgood,
1903]: a
thoroughgoing
account
is to be found in Siebenmann
\cite[\S 4, {\sc Historical notes.}]{Sie}, where the
contributions
coming from ``pure topology'' ([Schoenflies, 1906], [Tietze,
1913, 1914], [Antoine, 1921], [R.\,L. Moore, 1926], [Keldysh,
1966], \dots) are
analysed, and compared with those
coming from ``complex analytic methods'' ([Osgood, 1903],
[Carath\'eodory, 1913, 1913, 1913], [Koebe, 1913, 1913, 1915],
[Osgood-Taylor, 1913], [Study, 1913], \dots).
}.

So assume that $M$ is non-metric. By
Proposition~\ref{generalised_Jordan},
$M-J$ has two components, one of which must be non-metric. [If
both components were metric, then $M$ could be expressed as
the union of those plus $J$ so would be Lindel\"of, hence
metric.]
Pick a non-metric component of $M-J$,
and call it the {\it exterior of $J$} (denoted by $J_{\rm
ext}$). Call the other component the {\it interior of $J$}
(denote it $J_{\rm int}$). Define $W_{\rm int}$ and $W_{\rm
ext}$ by adding $J$ to $J_{\rm int}$ and $J_{\rm ext}$
respectively. It is easy to check that both these $W$'s are
surfaces-with-boundary (this is a local question which
can be handled via the metric version of
(ST), compare Lemma~\ref{Siebenmann}~(ii)).

The sequel
depends on the following
homological compactness criterion:

\begin{lemma}\label{compactness-criterion} A connected Hausdorff
surface-with-boundary $W$
such that $H_1(W)=0$ and with boundary $\partial W \approx
{\Bbb S}^1$ is compact.
\end{lemma}

\begin{proof}  Notice first that the conclusion is easy to corrupt
without Hausdorff: consider a 2-disc with infinitely many
origins (which is not quasi-compact).
%
Recall from Samelson~\cite[Lemma D]{Sa} that {\it a connected
Hausdorff noncompact $n$-manifold $M^n$ has a vanishing
top-dimensional (singular) homology, i.e. $H_n(M)=0$.} (Note
that Samelson's proof does not employ any metric assumption.)

Consider the double $2W=:M=W\cup W_{\ast}$, where $W_{\ast}$
is
a copy of $W$. By the Mayer-Vietoris sequence:
$$
\dots \to H_2(W)\oplus H_2(W_{\ast}) \to H_2(M) \to
H_1(\partial W=W\cap W_{\ast})\to H_1(W)\oplus
H_1(W_{\ast})\to \dots
$$
Since the last groups are zero by assumption, $H_2(M)$
surjects onto the nontrivial $H_1(\partial W)$, so is itself
non-zero. By the aforementioned (Samelson's Lemma D) it
follows that $M$ is compact, hence $W$ is also compact
(because $W$ is closed in $M$).
\end{proof}

Now since $W_{\rm ext}$ is non-compact (else its interior
would be metric), Lemma~\ref{compactness-criterion}
implies $H_1(W_{\rm ext})\neq 0$.
Write the Mayer-Vietoris sequence of the decomposition
$M=W_{\rm int} \cup W_{\rm ext}$, and set $U=W_{\rm int}$ and
$V=W_{\rm ext}$ to simplify notation
(one should
work with open sets obtained by slight
collared enlargements of the two $W$'s):
$$
\dots \to H_2(U\cup V) \to H_1(U\cap V) \to H_1(U) \oplus
H_1(V) \to H_1(U\cup V) \to \dots \,.
$$
By Samelson's Lemma~D, we have $H_2(U\cup V)=0$ since $U\cup
V=M$ is non-metric hence non-compact.
Moreover $H_1(U\cup V)=0$ as $M$ is assumed to be
simply-connected.
So
exactness
gives
an isomorphism $H_1(U\cap V)
\approx H_1(U) \oplus H_1(V)$. Now as $U\cap V=J \approx{\Bbb
S}^1$, the first group is ${\Bbb Z}$. Recalling that
$H_1(V)\neq 0$, it follows (from the indecomposability of
${\Bbb Z}$ as a sum of abelian groups) that $H_1(U)=0$. A
second application of
Lemma~\ref{compactness-criterion} shows that $U=W_{\rm int}$
is compact. Summarising $U$ is a connected compact
surface-with-boundary with one boundary component and
$H_1(U)=0$ (so $\chi(U)=1-0+0=1$) and which is orientable
(being embedded in the simply-connected surface $M$). The
classification
of compact surfaces tell us that $U$
must be the $2$-disc,
which completes the proof.
\end{proof}

\section{
A converse
to
the non-metric Schoenflies theorem}

The purpose of this section is to provide a converse to
Proposition~\ref{generalised_schoenflies},
i.e. to show the following.

\begin{prop} \label{converse}
Suppose that $M$ is a Hausdorff surface. If each embedded
circle $J$ in $M$ bounds a $2$-disc in $M$ then $M$ is
simply-connected.
\end{prop}

\begin{proof}
It is enough to show that if $\lam\colon[0,1]\to M$ is a
non-constant loop in $M$ then $\lam$ is homotopic modulo
$\{0,1\}$ to an embedded circle. Suppose given such a loop
$\lam\colon [0,1]\to M$. As $\lam([0,1])$ is compact it may be
covered by finitely many coordinate charts, hence lies in a
metrisable surface. Like every metrisable surface, this
surface is the geometric realisation of a simplicial complex,
say $K$; see for example \cite[p.\,60]{M} or \cite{R}. We may
assume that $\lam(0)=\lam(1)$ is a vertex of $K$. By the
Simplicial Approximation Theorem, see for example
\cite[Theorem 1.6.11, p.\,31]{Ru}, $\lam$ is homotopic modulo
the base point to a simplicial approximation $\mu:[0,1]\to|K|$
to $\lam$ such that $\mu(0)=\mu(1)=\lam(1)$. Moreover, by
General Position, \cite[Theorem 1.6.10]{Ru} we may assume that
$\mu$ is in general position, so that its singular point set
is discrete. Thus there is a partition
$\{0=t_0<t_1<\dots<t_n=1\}$ of $[0,1]$ consisting solely of
the singular points of $\mu$.

For each $i=1,\dots,n$ either $\mu|[t_{i-1},t_i]$ is an
embedding or $\mu(t_{i-1})=\mu(t_i)$ and $\mu|[t_{i-1},t_i)$
is an embedding. In the latter case $\mu|[t_{i-1},t_i]$ is an
embedded circle so by hypothesis bounds a 2-disc in $M$. We
may use this 2-disc to find a homotopy fixing the end points
from $\mu$ to a loop which agrees with $\mu$ on
$[0,t_{i-1}]\cup[t_i,1]$ and is constant on $[t_{i-1},t_i]$,
then further homotope modulo the end points to a simplicial
map which agrees with $\mu$ on $[0,t_{i-2}]\cup[t_{i+1},1]$
and embeds $(t_{i-2},t_{i+1})$ onto
$\mu((t_{i-2},t_{i-1}]\cup[t_i,t_{i+1}))$ (with $t_{i-2}$
replaced by 0 if $i=1$ and $t_{i+1}$ replaced by 1 if $i=n$).
Repeating this procedure eventually we reach a loop $\nu$
which is homotopic to $\mu$, hence $\lam$, modulo the end
points and is such that $\nu([0,1])$ is an embedded circle, as
required.
\end{proof}

\section{
Dynamical applications of Jordan and Schoenflies}


Since non-metric manifolds cannot support minimal flows, it is
more reasonable to ask:
{\it which manifold admits a
transitive resp. a non-singular flow (in short a brushing)?}
The
well-known paradigms to the effect that Jordan separation
(dichotomy) obstructs transitivity, while Schoenfliesness
(more accurately non-vanishing Euler characteristic) impedes
brushability, extend beyond the metric (resp. compact) case.
Let us be more precise.

The non-metric Jordan theorem (Proposition~\ref{generalised_Jordan})
supplies
food to the following ``Bendixson type'' result:

\begin{prop} \label{Bendixson} A dichotomic surface
(i.e. divided by any embedded circle) cannot support a
transitive flow.
\end{prop}

\begin{proof} It is
a minor adaptation of the classical Bendixson bag argument.
Assume by contradiction that there is a point $x$ in the
surface $S$ with a dense orbit under a flow $f$. We may
draw a cross-section $\Sigma_x$ through $x$ and consider an
associated flow-box $f([-\varepsilon,\varepsilon] \times
\Sigma_x)$. Note that the Whitney--Bebutov theory classically
stated under a metric assumption
\cite[p.\,333]{Nemytskii-Stepanov}, holds
more universally, since the
orbit $f({\Bbb R}\times V)$ of a chart $V$ is
Lindel\"of\footnote{Recall the fact that the product of a
$\sigma$-compact with a Lindel\"of space is Lindel\"of.}. The
point $x$ must eventually return to $\Sigma_x$, and we call
$x_1$ its first return to $\Sigma_x$. The piece of trajectory
from $x$ to $x_1$ closed up by
the arc $A$ of $\Sigma_x$ joining $x$ to $x_1$ defines a
Jordan curve $J$ on $S$. It is easy to check that the
component of $S-J$ containing the near future of $x_1$ (e.g.
$f(\varepsilon / 2, x_1)$) contains in fact the full future of
$x_1$. Conclude by noticing that the ``short past'' of the arc
$A$ namely the set $f(]-\varepsilon,0[ \times {\rm int} A)$ is
an open subrectangle
which cannot intersect the orbit of $x$.
\end{proof}

In view of Proposition~\ref{generalised_Jordan} any simply-connected
surface is dichotomic, hence intransitive. Examples include
the (original) Pr\"ufer surface described in Rad\'o~\cite{R},
the Moore surface, the Maungakiekie surface (which is a plane
out of which emanates a long ray). A non simply-connected
example is the doubled Pr\"ufer surface $2P$ (of
Calabi-Rosenlicht, cf. e.g. \cite[Example~4.4]{BGG}), which is
clearly dichotomic, hence intransitive.

Schoenflies also has an obvious dynamical implication in
relation with its immediate successor Brouwer. Indeed on a
Schoenflies surface as soon as a flow line closes up into a
periodic orbit, a fixed point
is created somewhere (Brouwer's fixed-point theorem applied to
the bounding disc). Of course
Schoenfliesness alone is not enough to ensure the presence of
a periodic orbit (consider the plane ${\Bbb R}^2$ or the
semi-long plane ${\Bbb R} \times {\Bbb L}$ ``brushed'' along
the first factor). However the same condition of {\it
$\omega$-boundedness} as the one occurring in Nyikos' Bagpipe
theorem,
ensures that one
will find in the compact closure of an orbit
a minimal set (Zorn's lemma argument), which must be either a
point or a periodic orbit (by the Poincar\'e-Bendixson
argument). So picturesquely
the motion spirals towards a {\it cycle limite}.
Hence
we get:

\begin{prop} On an $\omega$-bounded,
Schoenflies (equivalently simply-connected) surface any flow exhibits a fixed point.
\end{prop}

This
may be regarded as a non-metric
pendant to the ``hairy ball theorem''
(the \hbox{2-sphere} cannot be foliated nor brushed). The
proposition applies for instance to the long plane ${\Bbb
L}^2$ (in which case an alternative proof may also be deduced
from the classification of foliations on ${\Bbb L}^2$ given in
\cite{BGG}). It also applies to any space obtained from a
Nyikos long pipe, \cite{Nyikos_84}, by capping off the short
end by a 2-disc, for example the {\it long glass}, i.e. the
semi-long cylinder ${\Bbb S}^1\times \mbox{(closed long ray)}$
capped off by a 2-disc.

A more
systematic study of the dynamics of  non-metric manifolds
should appear in a forthcoming paper \cite{BGG3}.

\def\contract{\vskip-8pt\penalty0}

\medskip

{
\hspace{+5mm} 
{\footnotesize
\begin{minipage}[b]{0.6\linewidth} Alexandre
Gabard

Universit\'e de Gen\`eve

Section de Math\'ematiques

2-4 rue du Li\`evre, CP 64

CH-1211 Gen\`eve 4

Switzerland

alexandregabard@hotmail.com
\end{minipage}
\hspace{-25mm}
\begin{minipage}[b]{0.6\linewidth}
David Gauld

Department of Mathematics

The University of Auckland

Private Bag 92019

Auckland

New Zealand

d.gauld@auckland.ac.nz
\end{minipage}}

}

\end{document}